\documentclass[11pt]{amsart}

\usepackage{pinlabel}
\usepackage{amsmath}
\usepackage{amssymb}
\usepackage{amsthm}
\usepackage[all,matrix]{xy}
\usepackage{graphicx}
\usepackage[bf,up,centerlast]{caption}
\usepackage[retainorgcmds]{IEEEtrantools}

\usepackage[retainorgcmds]{IEEEtrantools}

\newtheorem{Teo}{Theorem}[section]
\newtheorem{prop}[Teo]{Proposition}
\newtheorem{theo}[Teo]{Theorem}
\newtheorem{lemma}[Teo]{Lemma}
\newtheorem{coro}[Teo]{Corollary}

\theoremstyle{definition}
\newtheorem{definition}[Teo]{Definition}

\theoremstyle{remark}
\newtheorem{remark}[Teo]{Remark}

\newcommand{\bd}{\partial}

\newcommand{\F}{\mathcal{F}}
\newcommand{\Su}{\mathcal{S}}

\begin{document}


\title{Additivity of handle number and Morse-Novikov number\\
 of a-small knots}
\author{Fabiola Manjarrez-Guti\'errez}
\address{ \hskip-\parindent
Fabiola Manjarrez-Guti\'errez \\
 Centro de Investigaci\'on en Matem\'aticas\\
  Guanajuato, Guanajuato\\
 MX}
\email{fabiolamg@cimat.mx}
\thanks{Research supported by  UNAM and CONACyT}
\date{\today}

\begin{abstract}
A knot is an a-small knot if  its exterior does not contain closed incompressible surfaces disjoint from some incompressible Seifert surface for the knot. Using circular thin position for knots we prove that the handle number is additive under the connected sum of two a-small knots. As a consequence the Morse-Novikov number turns out to be additive under the connected sum of two a-small knots. 
\end{abstract}

\maketitle
\section{Introduction}
Let $K$ be a knot in $S^3$ and let $C_K= S^3- K$, the Morse-Novikov number of $K$, denoted by  $MN(K)$, was introduced in \cite{PRW} as the minimal possible number of critical points of a circle-valued Morse function $f: C_K \rightarrow S^1$ of a special type. In particular a knot $K$ is fiber if and only if $MN(K)=0$. It is also proved that the Morse-Novikov number is subadditive with respect to the connected sum of knots, $MN(K_1 \sharp K_2)\leq MN(K_1)+MN(K_2)$. They posed the question

\begin{center}
Is it true that $MN(K_1 \sharp K_2)= MN(K_1)+MN(K_2)$?
\end{center}

If $f: C_K \rightarrow S^1$ is of special type, then as in case of real-valued  Morse theory,  Goda observed in  \cite{Go3} that there is a correspondence between $f$ and a Heegaard splitting for the sutured manifold for a Seifert surface $R$ of $K$, with $R-K$ a regular level surface of $f$.  

The handle number  of $R$ is the number of 1-handles of the Heegaard splitting for the sutured manifold of $R$, while the Morse-Novikov number equals  the number of 1-handles and 2-handles. Hence for a knot $K$ in $S^3$  $MN(K)= 2 \times$ min$ \{$h(R); R is a Seifert surface for K$\}$. The handle number of a knot $K$ can be defined as $h(K)=$min$ \{$h(R); R is a Seifert surface for K$\}$.

Another theorem due to Goda (\cite{Go1} Theorem 2), says that the handle number of the $2n$-Murasugi sum $R_1 * R_2$ of two Seifert surfaces satisfies the inequality  $h(R_1) +h(R_2) -2(n-1)\leq h(R_1 * R_2)\leq h(R_1) + h(R_2)$. 

In \cite{M} the author studies circular handle decompositions for the exterior of a knot  which are also obtained from  circle-valued Morse maps.  

In this article we combine circular handle decomposition for knot exteriors and Heegaard splittings for sutured manifolds corresponding to knot exteriors  to prove that handle number of knots is additive under the connected sum of two a-small knots.

A knots is said to be a-small  if  its exterior does not contain closed incompressible surfaces disjoint from some incompressible Seifert surface for the knot. 

We prove the following theorem:

\begin{theo}
\label{teoremon1}
If $K=K_1 \sharp K_2$ is a connected sum of two a-small  knots, then $h(K)= h(K_1)+h(K_2)$.
\end{theo}

As a consequence we have that Morse-Novikov number is additive under connected sum of a-small knots.

\begin{coro}
\label{corolon}
If $K=K_1\sharp K_2$ is a connected sum of two a-small knots, then $MN(K)=MN(K_1)+MN(K_2)$.
\end{coro}

The paper is organized as follows: In section \ref{prelim} we review definitions concerning Morse-Novikov number, Heegaard splittings for sutured manifolds,  handle number and circular thin position. The concept of a-small knot is introduced in Section \ref{almostsmall}, we also study some properties of circular handle decompositions and Heegaard splittings for such knots. In Section \ref{principal} we prove that the handle number of an a-small knot is realized over an incompressible Seifert surface and we prove  Theorem \ref{teoremon1} and Corollary \ref{corolon}.

\section{Preliminaries}
\label{prelim}
\subsection{Morse-Novikov number and Handle number}
\label{subsec-handle}

In \cite{PRW},  Pajitnov, Rudolph and Weber introduced the concept of the \textit{ Morse-Novikov } number  of a knot $K\subset S^3$. A Morse map $f: C_K \rightarrow S^1$ is said to be \textit{regular}  if $K$ has a neighborhood framed as $S^1 \times D^2$ and such that $K \sim S^1 \times\{0\}$ and the restriction $f|: S^1 \times (D^2-\{ 0\}) \rightarrow S^1$ is given by $f((x,y))=y/ |y|$.

\begin{definition}
The \textit{Morse-Novikov number} of a knot, denoted by $MN(K)$, is the least possible number of critical points of a regular circle-valued Morse mapping $f:C_K \rightarrow S^1$. 
\end{definition}

In particular, a knot $K$ is fibered if and only if $MN(K)=0$.  

Let $m_i(f)$ denote the number of critical points of $f$ of index $i$. 

\begin{definition}
A  Morse map $f:C_K \rightarrow S^1$ is \textit{minimal} if it is regular and for each $i$, $m_i(f)$ is minimal possible among all regular maps homotopic to $f$. 

A regular Morse map $f:C_K \rightarrow S^1$ is said to be \textit{moderate} if
\begin{enumerate}
\item $m_0(f)=m_3(f)=0$
\item All critical values corresponding to critical points of the same index coincide.
\item $f^{(-1)}(x)$ is a connected Seifert surface for any regular value $x\in S^1$
\end {enumerate}
\end{definition}

Pajitnov, Rudolph and Weber proved that every knot has a minimal Morse map which is moderate. Moreover if $f$ is a regular Morse map realizing $MN(K)$, then $MN(K)=m_1(f)+m_2(f)$.

Goda \cite{Go3} pointed out that there is a handle decomposition which corresponds to a circle-valued Morse map, which he calls a Heegaard splitting for sutured manifolds. 

The concept of sutured manifold was defined in \cite{Ga}. It is a very useful tool in studying knots and links. We describe it briefly below.

\begin{definition}
A \textit{sutured manifold} $(M,\lambda)$ is a compact oriented 3-manifold M together with a subset $\lambda \subset \bd M$ which is a union of finitely many mutually disjoint annuli. For each component of $\lambda$, a \textit{suture}, that is, an oriented core circle is fixed, and $s(\lambda)$ denotes the set of sutures. Every component of $R(\lambda)= \bd M - Int \lambda$ is oriented so that the orientations on $R(\lambda)$ are coherent with respect to $s(\lambda)$, i.e., the orientation of each component of $\bd R(\lambda)$, which is induced by that of $R(\lambda)$, is parallel to the orientation of the corresponding component of $s(\lambda)$. Let $R_+ (\lambda)$ (resp. $R_- (\lambda)$) denotes the union of those components of $R(\lambda)$ whose normal vector point out of (resp. into) $M$. In the case that $(M,\lambda)$ is homeomorphic to $(F \times [0,1], \bd F \times [0,1])$ where $F$ is a compact oriented 2-manifold, $(M, \lambda)$ is called a \textit{product sutured manifold}.
\end{definition}

Let $K$ be an oriented knot in $S^3$, and $R'$ a Seifert surface for $K$. Set $R=R' \cap E(K)$, and $(P, \delta)=(N(R), N(\bd R))$. We will call $(P,\delta)$ a \textit{product sutured manifold for $R$}. Let $(M, \lambda)= (cl(E(K)-P), cl(\bd E(K)- \delta)$ with $R_{\pm}(\lambda)= R_{\mp}(\delta)$. We call $(M, \lambda)$ a \textit{complementary sutured manifold} for $R$, for short just sutured manifold of $R$.

\begin{definition}
Let $S$ be a 2-sided surface in a 3-manifold $M$. We say that $S$ is \textit{compressible} if there is a  2-disk $D \subset M$ such that $D \cap int(S) = \bd D$ does not bound a disk in $S$.  $D$ is a compressing disk for $S$. If $S$ is not compressible, it is said to be \textit{incompressible}. 

We say that  $S$  is \textit{strongly compressible} if there are  two compressing disks,  $D_1$ lying on the +side of $S$ and $D_2$ lying on the $-$side of $S$,  with $\bd D_1$ and $\bd D_2$ disjoint  essential closed curves in $S$. Otherwise we say that $S$ is \textit{weakly incompressible}.
\end{definition}

Given a compressible surface $S$ in a 3-manifold $M$ we can produce a ``simpler'' surface. Let $D$ be a compressing disk for $S$ and let $N(D)= D\times [0,1]$ be a regular neighborhood of $D$ in $M$, then $N(D) \cap S$ is an annulus contained in $S$ whose boundary components are copies of $\bd D$.
Let $S' = cl((S- A) \cup D\times\{ 0\} \cup D\times \{ 1\} $, $S'$ is the surface obtained from $S$ by compressing along $D$.     

Notice that the surface $S'$ is simplier than $S$ in the sense that if we take $1-\chi(S') < 1- \chi(S)$. It is worth to point out that compressing a surface can result into disconnected pieces.

\begin{definition}
A compression body is a cobordism rel $\bd$ between surfaces $\bd_{+}W$ and $\bd_{-}W$ such that $W= \bd_{+}W \times I \cup$ 2-handles $\cup$ 3-handles and $\bd_{-}W$ has no sphere components.  We can see that if $\bd_{-}W \neq \emptyset$ and $W$ is connected, W is obtained from $\bd_{-}W \times I$ by attaching a number of 1-handles along  disks on $\bd_{-}W \times \{1 \}$, where $\bd_{-}W$ corresponds to $\bd_{-}W \times \{ 0\}$.

We denote by $h(W)$ the number of these 1-handles.
\end{definition}

By the construction of a compression body is  not hard to check that $\bd_- W$ is an incompressible surface in $W$.

\begin{definition}
$V \cup W$ is a Heegaard splitting for $(M, \lambda)$ if:
\begin{enumerate}
\item $V, W$ are connected compression bodies.
\item $V\cup W = M$.
\item $V\cap W= \bd_{+}V=\bd_{+}W$ , $\bd_{-}V= R_{+}(\lambda)$ and $\bd_{-}W= R_{-}(\lambda)$ 
\end{enumerate}

We say that $V \cap W= S$ is a \textit{Heegaard surface of} $V\cup W$. Then $\bd S= \bd(\bd_+ V)= \bd(\bd_+ W)= s(\lambda)$.

A Heegaard splitting usually will be denoted by $V \cup_S W$.

The \textit{genus of a Heegaard splitting}, denoted by  $g(V\cup_S W)$, is defined to be the genus of the Heegaard surface $S$.

\end{definition}

Let $K$ be a knot in $S^3$ and $R$ a Seifert surface for $K$. Let $(M, \lambda)$ be the sutured manifold for $R$.

\begin{definition}
Set $h(R)=$min$\{h(V); V \cup W$ is a Heegaard splitting for  $(M, \lambda)\}$.  We call $h(R)$  \textit{the handle number of} $R$. 
\end{definition}

The handle number is an invariant of a Seifert surface. In the papers \cite{Go1} and \cite{Go2} Goda develops efficient methods to compute the handle number of a Seifert surface for relatively simple knots. He shows that every non-fibered knot with at most 10 crossings has a minimal genus Seifert surface whose handle number is 1.

The handle number  of $R$ is the number of 1-handles of the Heegaard splitting for the sutured manifold of $R$ , while the Morse-Novikov number equals  the number of 1-handles and 2-handles. Hence we have the following definition.

\begin{definition}
The \textit{handle number of a knot} is defined to be  $h(K)=$ min$ \{$h(R); R is a Seifert surface for K$\}$.
\end{definition}

Thus we have $MN(K)=2 \times h(K)$.

\begin{definition}
A sutured manifold $(M, \lambda)$ is \textit{$\bd$-reducible} if any component of $R(\lambda)$ is compressible.

A Heegaard splitting $V\cup_S W$ for $(M, \lambda)$ is \textit{$\bd$-reducible} if there is a compressing for $R(\lambda)$ which intersects $S$ in a single curve.
\end{definition}

For 3-manifolds it is known that any Heegaard splitting of a $\bd$-reducible manifold is $\bd$-reducible, see for instance \cite{S}. Analogous we have:

\begin{prop}
\label{reducible}
Any Heegaard splitting of a $\bd$-reducible sutured manifold is $\bd$-reducible.
\end{prop}

\begin{definition}
A Heegaard splitting $V\cup_S W$ for $(M, \lambda)$ is said to be \textit{weakly reducible} if there exist essential  disks $D_1\subset V $ and $D_2 \subset W$ so that $\bd D_1$ and $\bd D_2$ are disjoint in $S$ ($S$ is strongly compressible). 

If $V \cup_S W$ is not weakly reducible we say it is \textit{strongly irreducible} ($S$ is weakly incompressible).
\end{definition}

\begin{remark}
If a Heegaard splitting $V\cup_S W$ is weakly reducible then the surface $S$ can be compressed simultaneously in both directions, that is, both into $V$ and simultaneously into $W$.

Let $\Delta_1 \subset V$ and $\Delta_2 \subset W$ be collections of essential disks in the respective compression bodies so that $\bd \Delta_1$ and $\bd \Delta_2$ are disjoint in $S$ and the families $\Delta_i$ are maximal with respect to this property.  That is, if  $S_1$ ($S_2$) represents  the surface in $V$ ($W$) obtained by compressing $S$ along $\Delta_1$ ($\Delta_2$), then any further compressing disk of $S_1$ ($S_2$) into $V$($W$) will necessary have boundaries intersecting the  boundaries of the other disk family.

Let $\bar{S}$ be the surface obtained by compressing $S_1$ along $\Delta_2$ (or symetrically, $S_2$ along $\Delta_1$). The surfaces $S_1$, $S_2$ and $\bar{S}$ can be pushed away to be disjoint. $\bar{S}$ separates $M$ into the remnant $H_1$ of $V$ and the remnant $H_2$ of $W$. Each component of $H_i$ inherits a Heegaard splitting surface, namely a component of $S_i$. This splitting itself may be weakly reducible and we can continue the process. Ultimately a Heegaard splitting is thereby broken up into a series of strongly irreducible splittings.  (See \cite{ST}).
\end{remark}

The above process will be referred as \textit{weak reduction} of  Heegaard splitting. After performing weak reductions the surfaces $S_i$ and $\bar{S}$ can be disconnected.

\begin{definition}
A \textit{generalized Heegaard splitting of a sutured manifold} $(M, \lambda)$ is a structure:

$(V_1 \cup_{S_1} W_1) \bigcup_{F_1} (V_2 \cup_{S_2} W_2) \bigcup_{F_2} .... \bigcup_{F_{m-1}} (V_m \cup_{S_m} W_m)$

Each of $V_i$ and $W_i$ are compression bodies, $\bd_{+}V_i=S_i=\bd_{+}W_i$, $\bd_{-}W_i=F_i=\bd_{-}V_{i+1}$, $\bd_{-}V_1= R_{+}(\lambda)$, $\bd_{-}W_m=R_{-}(\lambda)$, $\bd S_i \sim s(\lambda)$, $\bd F_i \sim s(\lambda)$.

The surfaces $F_i$'s are called \textit{thin surfaces} and the $S_i$'s \textit{thick surfaces}.

($V_i \cup W_i$ is a union of Heegaard splittings of a submanifold of $(M, \lambda)$).

A generalized Heegaard splitting is strongly irreducible if each of the $V_i\cup W_i$ is strongly irreducible.
\end{definition}

Given  a weakly reducible  Heegaard splitting we can obtain a generalized Heegaard splitting as explained in the above remark. 

The inverse process is also of interest, given a generalized Heegaard splitting we can obtain a Heegaard splitting. This was introduced in \cite{Sc}.

\begin{definition}
The following process is called \textit{amalgamation}.
Let $(V_1 \cup_{S_1} W_1) \bigcup_{F_1} (V_2 \cup_{S_2} W_2) \bigcup_{F_2} .... \bigcup_{F_{m-1}} (V_m \cup_{S_m} W_m)$ be a generalized Heegaard splitting for $(M, \lambda)$, assume $m>1$. $W_1$ is a compression body that can be viewed as obtained from $F_1 \times [0, 1]$ by attaching some 1-handles to $F_1 \times \{ 0\}$. $V_2$ is a compression body that can be obtained from $F_1 \times [0,1]$ by attaching some 1-handles to $F_1 \times \{ 1\}$. The attaching disk of these 1-handles in $F_1 \times \{ 0\}$ and $F_1 \times \{ 1\}$ can be taken to project to disjoint disks in $F_1$. Collapse $F_1 \times [0, 1]$ to $F_1$. Then the 1-handles of $W_1$ are attached to $S_2= \bd_+ W_2$ which makes it a compression body $W_1'$, and the 1-handles of $V_2$ are attached to $S_1= \bd_+ V_1$ which makes it a compression body $V_1'$. Moreover $\bd_+ W_1'= \bd_+ V_1'$. Replacing $V_1$ and $V_2$ by $V_1'$ and $W_1$ and $W_2$ by $W_1'$ produces a new generalized Heegaard splitting in which $m$ is smaller. If we continue this process, we will eventually produce a Heegaard splitting $V\cup_S W$ for $(M, \lambda)$.
\end{definition}

A weakly reducible Heegaard splitting is a non-trivial amalgamation of a generalized Heegaard splitting.

\begin{remark}
\label{genusformula}
\begin{enumerate}
\item
The process of amalgamation gives a natural construction for pasting manifolds together.

Amalgamation of a Heegaard splitting of genus $n$ of a manifold $N$ and a genus $l$ Heegaard splitting of a manifold $L$ along boundary components $R \subset \bd N$ and $S \subset \bd L$ of genus $k$ has genus $n+l-k$.

\item
A Heegaard splitting can be viewed as a handle decomposition. Given $M=V \cup_S W$ there is a collection of handles such that $M= \bd_ -V\times [0,1] \cup N \cup T$, where $N$ denotes a collection of 1-handles and $T$ is a collection of 2-handles.  Consequently a generalized Heegaard splitting $(V_1 \cup_{S_1} W_1) \bigcup_{F_1} (V_2 \cup_{S_2} W_2) \bigcup_{F_2} .... \bigcup_{F_{m-1}} (V_m \cup_{S_m} W_m)$  has a description in terms of handles. For  each $i=1,2,..., m$, $V_i \cup_{S_i} W_i = (F_{i-1} \times [0,1]) \cup N_i \cup T_i$, where $F_0= \bd_- V_1$. 
\end{enumerate}
\end{remark}

Let $V \cup_G W$ a Heegaard splitting for $(M, \lambda)$.  Suppose  $F\times [0,1] \cup N \cup T$ a handle decompostion for $V \cup_G W$.  Let  $(V_1 \cup_{G_1} W_1) \bigcup_{F_2} (V_2 \cup_{G_2} W_2) \bigcup_{F_3} .... \bigcup_{F_{m}} (V_m \cup_{G_m} W_m)$ be a generalized Heegaard splitting of $V \cup_G W$ obtained by weak reductions, where $F=\bd_- V_1= \bd_- V$. Let $F \times [0,1] \cup N_1 \cup T_1 \cup.... \cup N_m \cup T_m$  a handle decomposition for the generalized Heegard splitting. Suppose that each $F_i$ and $G_i$ are connected for $i=1,2..., m$.

Using  the formula in Remark \ref{genusformula} and the handle decomposition of the Heegaard splittings we have:

\begin{IEEEeqnarray}{rCl} 
\label{genus1}
g(V \cup_G W) & = & g(G) \nonumber
\\ 
& = & g(F) + |N|
\end{IEEEeqnarray}

and

\begin{IEEEeqnarray}{rCl}
\label{genus2}
g(V \cup_G W) & = & \sum_{i=1}^m g(G_i)- \sum_{i=2}^m g(F_i) \nonumber
\\
& = & g(F) + \sum_{i=2}^m g(F_i) + \sum_{i=1}^m |N_i| - \sum_{i=2}^m g(F_i) \nonumber
\\
& = & g(F) + \sum_{i=1}^m |N_i|
\end{IEEEeqnarray}

From equations \ref{genus1} and \ref{genus2} we obtain:

\begin{equation*}
|N|= \sum_{i=1}^m |N_i|
\end{equation*}
The above equality can be interpreted by saying  that the number of 1-handles is invariant under amalgamation and under weak reduction. We have proved the following lemma:

\begin{lemma}
\label{handles}
Let $V \cup_G W$ a Heegaard splitting  for $(M,\lambda)$ and $(V_1 \cup_{G_1} W_1) \bigcup_{F_2} (V_2 \cup_{G_2} W_2) \bigcup_{F_3} .... \bigcup_{F_{m}} (V_m \cup_{G_m} W_m)$ be a generalized Heegaard splitting of $V \cup_G W$ obtained by weak reductions, such that $F_i$ and $S_i$ are connected for all $i=1,2,...m$. And let  $F\times [0,1] \cup N \cup T$ a handle decompostion for $V \cup_S W$ and $F \times [0,1] \cup N_1 \cup T_1 \cup.... \cup N_m \cup T_m$ the corresponding handle decomposition for the generalized Heegard splitting obtained from $V \cup_G W$ by weak reduction. Then $|N|=  \sum_{i=1}^m |N_i| $.

\end{lemma}

\subsection{Circular thin position}
\label{subsec-cir}
The author introduced circular thin position for knots  in \cite{M}.

Given a regular Morse function $f: C_K \rightarrow S^1$, as in the case of real-valued Morse functions, there is a correspondence between $f$ and a handle decomposition for $E(K)= S^3-N(K)$ the exterior of $K$, namely: 
\begin{center}
$E(K)= (F \times I)\cup N_1 \cup T_1 \cup N_2 \cup T_2 \cup...\cup N_k \cup T_k / F\times \{0\} \sim F\times \{1\}$, 
\end{center}

where $F$ is a  Seifert surface for $K$, $F-K$ is a regular level surface of $f$, $N_i$ is a collection of 1-handles corresponding to index 1 critical points, and $T_i$ is a collection of 2-handles corresponding to index 2 critical points.

We will call this decomposition a \textit{circular handle decomposition} for $E(K)$.

Let us denote by $G_i$ the surface  $cl (\partial (( F \times I)\cup N_1\cup T_1...\cup N_i )\setminus \bd E(K) \setminus F \times 0)$ and let $F_{i+1}$ be the surface $cl(\partial (( F \times I)\cup N_1\cup T_1...\cup T_i )\setminus \partial E(K) \setminus F \times 0)$, where $cl$  means the closure. When $i=k$, $F_{k+1}= F_1= F$. Every $G_i$ and $F_i$ contains a Seifert surface for $K$; note that $F_i$ or $G_i$ may be disconnected.

The surfaces $G_i$ and $F_i$, for $i=1,2,...,k$ will be called \textit{level surfaces}.

A level surface $F_i$ is called a \textit{thin surface} and a level surface $G_i$ is called a \textit{thick surface}. 

Let  $W_i=($collar of $F_i)\cup N_i\cup T_i$. $W_i$ is divided by a copy of $S_i$ into two compression bodies $A_i=($collar of $F_i)\cup N_i$ and $B_i=($collar of $G_i)\cup T_i$. Thus $G_i$ describes a Heegaard splitting of $W_i$ into compression bodies $A_i$ and $B_i$, where  $\partial_{-}A_1=F$, $\partial_{+} A_i= \partial_{+}B_i$, $\bd_{-}B_i=\bd_{-}A_{i+1}$ ($i=1,2,...,k-1$), $\partial_{-}B_k=F$. Thus we can write

$E(K)= A_1 \cup _ {G_1} B_1 \bigcup _ {F_2} A_2 \cup _ {G_2} B_2 \bigcup _ {F_3} ... \bigcup _ {F_{k}} A_k \cup _ {G_k} B_k /  F\times \{0\} \sim F\times \{1\}$.

This decomposition will be called a \textit{generalized circular Heegard splitting} ( or  \textit{gc-Heegaard splitting}). If $k=1$ we just call it a \textit{circular Heegaard splitting} (or \textit{c-Heegaard splitting}).

Figure  \ref{cirdec} shows a schematic picture of a circular handle decomposition with level surfaces and compression bodies indicated.

\begin{figure}[htp]
\labellist 
\small\hair 2pt 
\pinlabel $F_1$ at 183 88
\pinlabel $G_1$ at 153 157
\pinlabel $F_2$ at 87 184
\pinlabel $G_2$ at 19 156
\pinlabel $F_3$ at -9 87
\pinlabel $G_3$ at 17 26
\pinlabel $F_4$ at 94 -7
\pinlabel $G_4$ at 149 15
\pinlabel $A_1$ at 217 145
\pinlabel $B_1$ at 141 204
\pinlabel $W_1$ at 220 208
\pinlabel $N_1$ at 137 107
\pinlabel $T_1$  at 105 141
\pinlabel $N_2$ at 66 145 
\pinlabel $T_2$  at 30 109 
\pinlabel $N_3$ at 29 68
\pinlabel $T_3$  at 63 33
\pinlabel $N_4$ at 106 34
\pinlabel $T_4$ at 140 67 
\endlabellist 

\centering
\includegraphics[width=5cm]{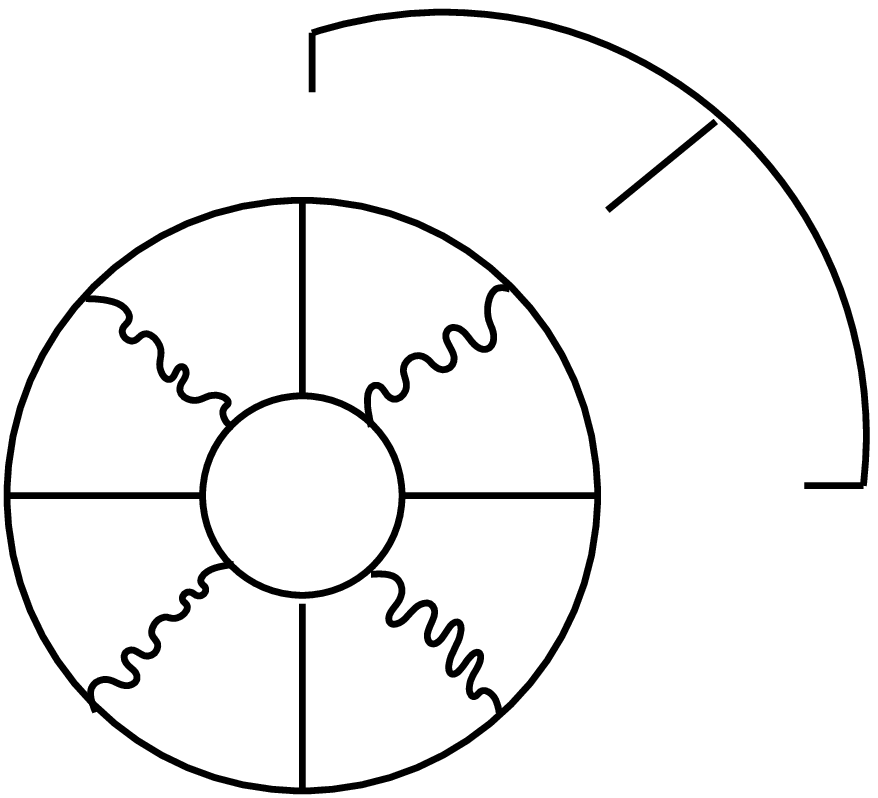}
\caption { Splitting of $E(K)$ into compression bodies}\label{cirdec}
\end{figure}

We wish to find a decomposition in which the $S_i$ are as simple as possible. 

\begin{definition}
For a compact connected surface $S$ different from $S^2$ or $D^2$ define the complexity of $S$, $c(S)$, to be  $c(S)=1-\chi(S)$.  If $S=S^2$ or $S=D^2$, set $c(S)=0$. If $S$  is disconnected we define $c(S)=\Sigma (c(S_i))$ where $S_i$ are the components of $S$.

Let $K$ be a knot in $S^3$.
Let $D$ be a circular handle decomposition for $E(K)$. Define the \textit{circular width of $E(K)$ with respect to the decomposition D , $cw(E(K),D)$}, to be the set of integers $\{ c(G_i),  1\leq i \leq k \} $. Arrange each multi-set  of integers in monotonically non-increasing order, and then  compare the ordered multisets lexicographically.

The \textit{circular width of $E(K)$}, denoted $cw(E(K))$, is the minimal circular width, $cw(E(K),D)$  over all possible circular decompositions $D$ for $E(K)$.

$E(K)$ is in \textit{circular thin position} if the circular width of the decomposition is the circular width of $E(K)$. 

If a knot $K$ is fibered we define the circular width of $K$, $cw(K)$, to be equal to zero.
\end{definition}

A nice property of a knot in circular thin position is that the thin surfaces are incompressible and the thick surfaces are weakly incompressible. For a proof of this fact see Theorem 3.2, \cite{M}. 

\begin{definition}
 A circular handle decomposition $D$ for a knot exterior $E(K)$ is called a \textit{ circular locally thin} decomposition if the thin level surfaces  $F_i$'s are incompressible and the thick level surfaces $G_i$'s are weakly incompressible.

\end{definition}

A circular (locally) thin decomposition gives raise to a \textit{strongly irreducible gc-Heegaard splitting}.

\begin{remark}
\ \
\begin{enumerate}

\item If $(M, \lambda)$ is the sutured manifold for $F$ a Seifert surface of a knot $K$ and  $(V_1 \cup_{G_1} W_1) \bigcup_{F_2} (V_2 \cup_{G_2} W_2) \bigcup_{F_3} .... \bigcup_{F_{m}} (V_m \cup_{G_m} W_m)$ is a generalized Heegaard splitting for $(M, \lambda)$, $m\geq 1$. After identifying  $\bd_-V_1 = R_+(\lambda)=F$ and $\bd_-W_m= R_-(\lambda)=F$ using the appropriate homeomorphims we recover the exterior of the knot, $E(K)$, and  it is provided with a gc-Heegaard splitting  $(V_1 \cup_{G_1} W_1) \bigcup_{F_2} (V_2 \cup_{G_2} W_2) \bigcup_{F_3} .... \bigcup_{F_{m}} (V_m \cup_{G_m} W_m) / F\times\{0\} \sim F \times \{1 \}$. $E(K)$ inherits  a circular handle decomposition as well. 

\item If $E(K)$ is provided with a gc-Heegaard splitting  $(V_1 \cup_{G_1} W_1) \bigcup_{F_2} (V_2 \cup_{G_2} W_2) \bigcup_{F_3} .... \bigcup_{F_{m}} (V_m \cup_{G_m} W_m) / F\times\{0\} \sim F \times \{1 \}$ , we can obtain a generalied Heegaard splitting for the sutured manifold for $F$ by cutting $E(K)$ along $F$.

\end{enumerate}
\end{remark}

Let us consider the knot exteriors $E(K_1)$ and $E(K_2)$. Assume they have the following circular handle decompositions:

 \begin{center}
 $E(K_1)= (F_1 \times I) \cup N_1 \cup T_1 \cup N_2 \cup T_2 \cup ... \cup N_n \cup T_n / F_1\times 0 \sim F_1 \times 1$ 
 \end{center} 

with level surfaces $F_1$, $G_1$, $F_2$..., $F_n$, $G_n$.

\begin{center}
$E(K_2)= (R_1 \times I )\cup O_1 \cup W_1 \cup O_2 \cup W_2 \cup ... \cup O_l \cup W_m / R_1 \times 0 \sim R_1 \times 1$
\end{center}

with level surfaces  $R_1$, $S_1$, $R_2$...,$R_m$, $S_m$.

Let $K= K_1 \sharp K_2$ be the connected sum of $K_1$ and $K_2$. There is a natural way to obtain a circular handle decomposition $D$ for $E(K)$ as follows. Let $R=F_1 \sharp R_1$ be a bounary connected sum of $F_1$ and $R_1$.  $R$ is a Seifert surface for $K$,  we attach the sequence of handles corresponding to $E(K_1)$, i.e., we attach $N_i$ and $T_i$, along the $F_1$ summand of $R$. Then we attach the sequence of handles corresponding to  $E(K_2)$, i.e., we attach  $O_j$ and $W_j$, along the $R_1$ component of $R$. The circular width of $D$,  $cw_D(E(K_1\sharp K_2))$, gives an upper bound for the circular width of $E(K_1\sharp K_2)$, namely  $cw(E(K_1 \sharp K_2)) \leq cw_D (E(K_1\sharp K_2))$. In \cite{EM} it is proved that the equality holds in some special cases.

The proof of that result relies on the following two  results also proved in \cite{EM}. Recall that for a connected sum of knots, $K_1\sharp K_2$, there is a decomposing sphere $\Sigma$ that intersects    $K_1\sharp K_2$ in two points. Let $A$ be the annulus in $E(K_1 \sharp K_2)$ given by $\Sigma \cap E(K_1\sharp K_2)$. 

\begin{prop}
\label{prop1}
Suppose that $E(K_1 \sharp K_2)$ is in circular (locally) thin position  with $\F$ the family of thin surfaces and $\Su$ the family of thick surfaces. Then $\F \cup \Su$ can be isotoped to intersect $A$ only in arcs that are essential in both $A$ and $\F \cup \Su$.
\end{prop}

\begin{coro}
\label{coro1}
Suppose  $K= K_1 \sharp K_2$ is  in circular (locally) thin position. Let $E(K)= (F \times I) \cup N_1 \cup T_1 \cup ... \cup N_m \cup T_m$
be a handle decomposition realizing a circular (locally) thin position. Let $\mathcal{N}$ be the collection of $N_i$'s, let $\mathcal{T}$ be the collection of $T_i$'s.  Then there  are subcollections $\mathcal{N}_1$ and $\mathcal{N}_2$ of $\mathcal{N}$ such that $\mathcal{N}_1 \cup \mathcal{N}_2 = \mathcal{N}$ and $\mathcal{N}_1 \cap \mathcal{N}_2 = \emptyset$, and  subcollections $\mathcal{T}_1$ and $\mathcal{T}_2$ of $\mathcal{T}$ such that $\mathcal{T}_1 \cup \mathcal{T}_2 = \mathcal{T}$ and $\mathcal{T}_1 \cap \mathcal{T}_2 = \emptyset$, such that $\mathcal{N}_i$ and $\mathcal{T}_i$ define a circular handle decomposition for $E(K_i)$, $i=1,2$.
\end{coro}

These results  allows us to push 1-handles and 2-handles away from the annulus $A$. Moreover a collection of 1-handles $N_i$ (or a collection of 2-handles $T_i$) can be pushed away from $A$ in such a way that $N_i$ (or $T_i$) is totally contained in $E(K_j)\cap E(K_1\sharp K_2)$, for some $j=1,2$.

In other words, a circular (locally) thin decomposition for $E(K_1\sharp K_2)$ induces circular locally thin decompositions for $E(K_1)$ and $E(K_2)$. 

Another consequence is the following:

\begin{coro}
\label{coro2}
If $K=K_1 \sharp K_2$ has a circular (locally) thin position of the form $E(K)=(R \times I) \cup N_1 \cup T_1 / (R\times 0 \sim R \times 1)$,  then either $K_1$ or $K_2$ is fibered, say $K_1$, and $K_2$ is  not fibered.
\end{coro}

\section{Almost small knots}
\label{almostsmall}
In Subsection \ref{subsec-handle} we introduced weak reduction for Heegaard splittings, after performing this operation we may obtain a generalized Heegaard splitting with non connected level surfaces. In the definition of circular thin position for the exterior of a knot we noticed that the level surfaces may be disconneted, see Subsection \ref{subsec-cir}.

For our purposes we need knots whose circular (locally) thin decompositions contain connected level surfaces. Thus we introduce the following definition.

\begin{definition}
A knot $K$ in $S^3$ is \textit{almost small} (or \textit{a-small}) if the exterior $E(K)$ does not contain closed incompressible surfaces disjoint from some incompressible Seifert surface of $K$.\end{definition}

Small knots are almost small knots. The level surfaces of a  circular (locally) thin position for an a-small knot do not contain closed components.

\begin{lemma}
\label{lemita1}
Let $K$ be an a-small knot and suppose that $E(K)$ is in circular (locally) thin position. Then the level surfaces does not contain closed components.
\end{lemma}
\begin{proof}
Let  $E(K)= (F_1 \times I) \cup N_1 \cup T_1 \cup N_2 \cup T_2 \cup ... \cup N_n \cup T_n / F_1\times 0 \sim F_1 \times 1$ be a circular (locally) thin decomposition, with level surfaces $F_1$, $G_1$, $F_2$..., $F_n$, $G_n$. By construction $F_i \cap F_j= \emptyset$ and $G_i \cap G_j = \emptyset$ for all $i \neq j$. Suppose that $F_i$ contains a closed component $F'$, $F'$ is incompressible  and by construction $F' \cap F_j = \emptyset$ for all $j\neq i$, this contradicts the fact that $K$ is almost small. Therefore the thin levels do not contain closed components, in other words a thin level is connected. Any level surface $G_i$ is obtained from $F_i \times [0,1]$ by attaching 1-handles to $F_i \times \{ 1\}$, since $F_i$ is connected then $G_i$ is connected, thus $G_i$ does not contain closed components.
\end{proof}

A weakly reducible c-Heegaard splitting of an a-small knot gives raise to a strongly irreducible gc-Heegaard splitting after weak reductions for which the level surfaces are connected.

\begin{lemma}
Let $K$ be an a-small knot  and let $F$ be a Seifert surface for $K$.  Suppose  $E(K)= V\cup_G W / F \times \{ 0\} \sim F \times \{1 \}$ is a c-Heegaard splitting for $E(K)$ which is weakly reducible and let  $E(K)= A_1 \cup _ {G_1} B_1 \bigcup _ {F_2} A_2 \cup _ {G_2} B_2 \bigcup _ {F_3} ... \bigcup _ {F_{k}} A_k \cup _ {G_k} B_k /  F\times \{0\} \sim F\times \{1\}$ be a 
 strongly irreducible gc-Heegaard splitting obtained from  $V\cup_G W / F \times \{ 0\} \sim F \times \{1 \}$ after weak reductions. Then the surfaces $F_i$ and $G_i$ do not contain closed components for all $i=1,2,..., k$ 
\end{lemma}
\begin{proof}
The strongly irreducible gc-Heegaard splitting $E(K)= A_1 \cup _ {G_1} B_1 \bigcup _ {F_2} A_2 \cup _ {G_2} B_2 \bigcup _ {F_3} ... \bigcup _ {F_{k}} A_k \cup _ {G_k} B_k /  F\times \{0\} \sim F\times \{1\}$ 
gives raise to a circular (locally) thin decomposition for $E(K)$, by Lemma \ref{lemita1} the level surfaces of such decomposition do not contain closed components, which proves the lemma.
\end{proof}

The property of being a-small is preserved under connected sum.

\begin{lemma}
Let $K_1$ and $K_2$ be  a-small knots, then the knot $K_1 \sharp K_2$ is a-small.
\end{lemma}
\begin{proof}
Let $K$ be the connected sum of $K_1$ and $K_2$ and let $F$ be a closed incompressible surface in $E(K)$.  The exterior of $K$ can be seen as  $E(K)= E(K_1) \cup_A E(K_2)$, where $A$ is a separating annulus.  

If $F \cap A = \emptyset$, then either $F$ is contained in $E(K_1)$ or in  $E(K_2)$, say $F$ is in $E(K_1)$. Since $K_1$ is a-small then $F$ intersects every  incompressible Seifert surface of $K_1$. Let $S$ be an incompressible Seifert surface of $K$, we can view $S$ as the boundary connected sum of an incompressible  Seifert surface $S_1$  for $K_1$ and an incompressible Seifert surface $S_2$ of $K_2$, i.e , $S= S_1 \sharp S_2$. Moreover $F \cap S_1 \neq \emptyset$ and this implies $F \cap S \neq \emptyset$. Thus $K$ is a-small.

If $F \cap A \neq \emptyset$  and $F \cap A$ consists of essential closed curves in $A$. An incompressible Seifert surface for $K$ intersects $A$ in arcs connecting different boundary components of $A$, therefore $F \cap S \neq \emptyset$. Thus $K$ is a-small.
\end{proof}

We can apply  Lemma \ref{handles} to case of an a-small knot to obtain the following corollary:

\begin{coro}
\label{handles1}
Let us consider $K$ an a-small knot and $F$ a Seifert surface for $K$. Let $V \cup_G W$ a Heegaard splitting for $(M, \lambda)$ the sutured manifold for $F$. Suppose  $F\times [0,1] \cup N \cup T$ a handle decompostion for $V \cup_S W$.  Let  $(V_1 \cup_{G_1} W_1) \bigcup_{F_2} (V_2 \cup_{G_2} W_2) \bigcup_{F_3} .... \bigcup_{F_{m}} (V_m \cup_{G_m} W_m)$ be a strongly irreducilbe generalized Heegaard splitting of $V \cup_G W$ obtained by weak reductions and let $F \times [0,1] \cup N_1 \cup T_1 \cup.... \cup N_m \cup T_m$  a handle decomposition for the generalized Heegard splitting.

Then $|N|=  \sum_{i=1}^m |N_i| $.
\end{coro}


\section{Additivity of handle number for a-small knots}
\label{principal}
In this section we prove that handle number of an a-small knot is realized over an incompressible Seifert surface. Later on we will prove that handle number is additive under connected sum of a-small knots.

The following two lemmas work for any kind of knot in $S^3$.

\begin{lemma}
\label{lemus1}
Let $K$ be a knot and let $F$ be a compressible  Seifert surface for $K$. Suppose $E(K)$ has a circular handle decomposition $(F \times [0,1]) \cup N \cup T / F \times \{0\} \sim F \times \{1\}$ and let $V\cup_G W / F \times \{0\} \sim F \times \{1\}$ be the corresponding c-Heegaard splitting, then there is a compressing disk for $F$ that intersects $S$ in exactly one essential curve.
\end{lemma}
\begin{proof}
Let $D$ be a compressing disk for $F$ in $E(K)$. Notice that $F \cap D = \bd F$ since $F =\bd_- V$ ($F= \bd_- W$) for  the compression body $V$($W$) and $\bd_- V$ ($\bd_- W$) is incompressible in $V$($W$). Let $(M, \lambda)$ be the sutured manifold for $F$,  then $V\cup_G W$ is a Heegaard splitting for $(M, \lambda)$. This manifold is $\bd$-reducible, by Proposition \ref{reducible} the Heegaard splitting is $\bd$-reducible, in other words there is a boundary reducing disk $D'$ which  intersects the Heegaard surface $G$ exactly in one essential curve. Glueing back together the copies of $F$ in $(M, \lambda)$, we recover $E(K)$ and the disk $D'$ is the one required by the lemma.
\end{proof}

\begin{lemma}
\label{lemus2}
Let $K$ be a knot and let $F$ be a compressible  Seifert surface for $K$. Suppose $E(K)$ has a circular handle decomposition $(F \times [0,1]) \cup N \cup T / F \times \{0\} \sim F \times \{1\}$, where $N\neq \emptyset$ and let $V\cup_G W/ F \times \{0\} \sim F \times \{1\}$ be the corresponding c-Heegaard splitting, then $V \cup_G W / F \times \{0\} \sim F \times \{1\}$ is weakly reducible.
\end{lemma}
\begin{proof}
By Lemma \ref{lemus1}, the sutured manifold for $F$, $(M,\lambda)$, is $\bd$-reducible and the inherited Heegaard splitting $V \cup_G W$ is $\bd$-reducible. Then there is a compressing disk $D$ for $F$ which intersects $S$ in a single curve. We can assume that $\bd D$ is contained in $\bd_ -V$, then $V\cap D$ is an annulus $A$ with one boundary on $\bd_ - V$ and the other on $\bd_+ V$. Since $V$ is not trivial there is a properly embedded disk $D'$ in $V$ disjoint from $A$.  Let $D''= W \cap D$, $D''$ is a disk properly embedded in $W$ with $\bd D' \cap \bd D''= \emptyset$. Thus $V\cup_S W$ is weakly reducible. When we recover $E(K)$ the c-Heegaard splitting remains weakly reducible.
\end{proof}

\begin{theo}
Let $K$ be an a-small knot in $S^3$, then there is an incompressible Seifert surface $F$ for $K$ such that $h(K)=h(F)$.
\end{theo}
\begin{proof}
Let $F'$ be a Seifert surface for $K$ such that $h(F')=h(K)$, if $F'$ is incompressible there is nothing to prove. Suppose that $F'$ is compressible, let $V \cup_G W$ be the Heegaard splitting for the sutured manifold for $F'$ such that $h(V)=h(F')$, by Lemma \ref{lemus2} the Heegaard splitting  is weakly reducible. Then we can obtain a generalized Heegaard spliting  $(V_1 \cup_{G_1} W_1) \bigcup_{F_2} (V_2 \cup_{G_2} W_2) \bigcup_{F_3} .... \bigcup_{F_{m}} (V_m \cup_{G_m} W_m)$ in which $F_i$ are incompressible, except for $F_1=F'$ and $G_i$ are weakly incompressible and all $F_i's$ and $G_i's$ are connected. By Corollary \ref{handles1} the number of 1-handles for this generalized Heegaard splitting is equal to $h(V)$.

Glueing back together $F' \times \{0\}$ with  $F' \times \{1\}$ we obtain a strongly irreducible gc-Heegaard splitting for $E(K)$.  Let us fix $F=F_{i_0}$ for some $i_0 \in \{2,3, ..., m \}$. Open $E(K)$ along $F$, thus the sutured manifold for $F$ is provided with a strongly irreducible generalized Heegaard splitting. Notice that the number of 1-handles for this generalized Heegaard splitting has not been changed. After amalgamating we  obtain a Heegaard splitting $V' \cup_{G'} W'$ with $\bd_- V' = F= \bd_-{W'}$. It follows from Corollary \ref{handles1} that  $h(V)=h(V')$. Thus $h(F)= h(V)=h(K)$ and $F$ is incompressible as required.

\end{proof}

We proceed to prove the main Theorem.

\begin{theo}
If $K=K_1 \sharp K_2$ is a connected sum of two a-small knots, then $h(K)= h(K_1)+h(K_2)$.
\end{theo}

\begin{proof}
We will prove that both inequalities $h(K)\leq h(K_1)+h(K_2)$ and $h(K)\geq h(K_1)+h(K_2)$ hold.

Let $K_1$ and $K_2$ be a-small knots with handle number $h(K_1)$ and $h(K_2)$ respectively. Let $F_1$ and $F_2$ be Seifert surfaces realizing such numbers, i.e. $h(F_i)=h(K_i)$ for $i=1,2$.
The sutured manifold $(M, \lambda_i)$ for $F_i$ has a Heegaard splitting $V_i \cup_{G_i} W_i$ such that $h(V_i)= h(F_i)$. The corresponding c-Heegaard splitting for $E(K_i)$ gives raise to a circular handle decomposition $E(K_i)= F_i \times[0,1] \cup N_i \cup T_i / F_i \times \{ 0\} \sim F_i \times \{ 1\}$, such that $|N_i|=h(F_i)$.

The knot $K$ has a circular handle decomposition inherited from $K_1$ and $K_2$ given by $E(K)= F \times [0,1] \cup N_1 \cup T_1 \cup N_2 \cup T_2 / F \times \{ 0\} \sim F \times \{ 1\}$ where $F$ is homeomorphic to $F_1 \sharp F_2$. The sutured manifold $(M, \lambda)$ for $F$ inherits a generalized Heegaard splitting from such circular handle decomposition, namely $V'_1 \cup_{S_1} W'_1 \cup_{R'} V'_2 \cup_{S_2} W'_2$, which we amalgamate to obtain $V'\cup_{S'} W'$. By Lemma \ref{handles} we have $|N_1| + |N_2|= h(V')$, this gives an upper bound for $h(F)$, i.e, $h(F) \leq h(F_1)+ h(F_1) $. Thus $h(K) \leq h(K_1) + h(K_2)$.

Now, let $F$ be the incompressible Seifert surface for $K$ such that $h(K)=h(F)$. The sutured manifold $(M, \lambda)$ for $F$ has a Heegaard splitting $V \cup_G W$ such that $h(V)=h(F)$.

If $V \cup_G W$ is strongly reducible, then it corresponds to a circular locally thin decompositon for $E(K)$ of the form $E(K)= F\times [0,1] \cup N \cup T / F \times \{ 0\} \sim F \times \{ 1\}$, by Corollary \ref{coro2} either $K_1$ or $K_2$ is fibered, say $K_1$ is fibered and $K_2$ is not fibered. $E(K_2)$ inherits  a circular decomposition $E(K_2)=(F_2 \times I) \cup N \cup T / (F_2\times \{0\} \sim F_2 \times \{1\})$, thus $h(K_2) \leq h(K)$ and $h(K_1)=0$. Therefore $h(K_1)+ h(K_2) \leq h(K)$.

If $V \cup_G W$ is weakly reducible, we perform weak reductions to obtain a strongly irreducible generalized Heegaard splitting for $(M, \lambda)$, let  $F \times [0,1] \cup N_1 \cup T_1 \cup... \cup N_m \cup T_m$ be the corresponding handle decomposition for $(M, \lambda)$, we see that $h(V)= \sum_{i=1}^m |N_i|$. $E(K)$ inherits a circular handle decomposition $E(K)=F \times [0,1] \cup N_1 \cup T_1 \cup... \cup N_m \cup T_m/ F\times\{0\} \sim F\times\{1 \}$ which is locally thin.

By Proposition \ref{prop1} and Corollary  \ref{coro1} we obtain circular handle decompositions for $E(K_1)$ and $E(K_2)$.

Thus;

$E(K_1)= F_1 \times I \cup N^1_1 \cup T^1_1 \cup ... \cup N^1_n \cup T^1_n / F_1 \times \{0\} \sim F_1 \times \{ 1\}$.

$E(K_2)= F_2 \times I \cup N^2_1 \cup T^2_1 \cup ... \cup N^2_l \cup T^2_l / F_2 \times \{0\} \sim F_2 \times \{ 1\}$.

Where $F$ is homeomorphic to  $F_1 \sharp F_2$ and $\sum_{i=1}^n  |N^1_i |+ \sum_{j=1}^l | N^2_j| = \sum_{r=1}^m | N_r| = h(K) $.

The decompositions above imply  $h(K_1) \leq \sum_{i=1}^n  | N^1_i|$ and $h(K_2) \leq \sum_{j=1}^m  | N^2_j|$. Adding these inequalities we get $h(K_1) + h(K_2) \leq  h(K)$. 

 We have proved that $h(K)= h(K_1) + h(K_2)$.
\end{proof}

We observe that additivity of Morse-Novikov number is a corollary of the above result.

\begin{coro}
If $K=K_1\sharp K_2$ is a connected sum of two a-small knots, then $MN(K)=MN(K_1)+MN(K_2)$.
\end{coro}

\section*{Acknowledgements}
I want to thank Mario Eudave-Mu\~noz for valuable discussions and for suggesting the concept of a-small knot.

Part of this work was done while the author held a postdoctoral scholarship at Instituto de Matem\'aticas UNAM.

\end{document}